\documentclass[12pt,a4paper]{amsart}
\usepackage{graphicx,amssymb}
\usepackage{amsmath,amssymb,amscd}
\input xy
\xyoption{all}
\usepackage[all]{xy}
\usepackage{hyperref}
\newcommand{\bs}[1]{\boldsymbol{#1}}
\newcommand{\lra}{\longrightarrow}

\newcommand{\ra}{\rightarrow}

\newcommand{\ZZ}{\mathbb Z}
\newcommand{\QQ}{\mathbb Q}

\newcommand{\FF}{\mathbb F}

\newcommand{\cE}{\mathcal{E}}

\newcommand{\Ker}{\mbox{Ker}}

\newcommand{\Hom}{\mbox{Hom}}

\newcommand{\Jac}{\mbox{Jac}}

\newcommand{\ch}{\operatorname{ch}}
\theoremstyle{plain}
\newtheorem{theorem}{Theorem}[section]
\newtheorem{lem}[theorem]{Lemma}
\newtheorem{prop}[theorem]{Proposition}
\newtheorem{cor}[theorem]{Corollary}

\newtheorem{rem}[theorem]{Remark}
\newtheorem{ex}[theorem]{Example}
\numberwithin{equation}{section}

\begin{document}
\title[Brill-Noether loci in rank 2]{Non-emptiness of Brill-Noether loci in $M(2,L)$}

\author{H. Lange}
\author{P. E. Newstead}
\author{V. Strehl}

\address{H. Lange\\Department Mathematik\\
              Universit\"at Erlangen-N\"urnberg\\
              Cauerstrasse 11\\
              D-$91058$ Erlangen\\
              Germany}
              \email{lange@mi.uni-erlangen.de}
\address{P. E. Newstead\\Department of Mathematical Sciences\\
              University of Liverpool\\
              Peach Street, Liverpool L69 7ZL, UK}
\email{newstead@liv.ac.uk}              
\address{V. Strehl\\Department Informatik\\
              Universit\"at Erlangen-N\"urnberg\\
              Martensstrasse 3\\
              D-$91058$ Erlangen\\
              Germany}
\email{volker.strehl@informatik.uni-erlangen.de}

\date{\today}

\thanks{The first and second authors are members of the research group VBAC (Vector Bundles on Algebraic Curves). The second author 
would like to thank the Department Mathematik der Universit\"at 
         Erlangen-N\"urnberg for its hospitality}
\keywords{Stable vector bundle, Brill-Noether locus, Porteous formula}
\subjclass[2010]{Primary: 14H60; Secondary: 14F05}

\begin{abstract}
Let $C$ be a smooth projective complex curve of genus $g \geq 2$. We investigate the Brill-Noether locus consisting of stable bundles of 
rank 2 and determinant $L$ of odd degree $d$ having at least $k$ independent sections. This locus possesses a virtual fundamental class. 
We show that in many cases this class is non-zero, which implies that the Brill-Noether locus is non-empty. For many values of $d$ and $k$
the result is best possible. We obtain more precise results for $k\le5$. An appendix contains the proof of a combinatorial lemma which we need.

\end{abstract}
\maketitle

\section{Introduction}\label{intro}

Let $C$ be a smooth projective complex curve of genus $g \geq 2$. Let $M(2,d)$ be the moduli space of stable bundles of rank $2$ and degree $d$ and, for any line bundle 
$L$ of degree $d$, let $M(2,L)$ denote the moduli space of stable bundles of rank 2 and determinant $L$. The Brill-Noether locus $B(2,d,k) \subset M(2,d)$ is defined by
$$
B(2,d,k) := \{ E \in M(2,d) \;|\; h^0(E) \geq k \}.
$$
Similarly
$$
B(2,L,k) := B(2,d,k) \cap M(2,L).
$$
If $d \leq k+2g -2$, then $B(2,d,k)$ is a degeneracy locus whose expected dimension is
$$
\beta(2,d,k) := 4g-3 - k(k-d+2g -2).
$$
Similarly $B(2,L,k)$ is a degeneracy locus whose expected dimension is
$$
\beta(2,d,k) -g = 3g-3 - k(k-d+2g-2).
$$

A great deal is known about $B(2,d,k)$ (see for example \cite{te3} and more recently \cite{cf1} and \cite{cf2}; also \cite{te1} and \cite{bmno} for the case of general rank). Much less is known about $B(2,L,k)$, except when $L=K$, where $K$ is the canonical bundle on $C$ (see \cite{lnp} for a recent result and further references). In \cite{te2} Teixidor 
obtained a sufficient condition for $B(2,L,k)$ to be non-empty and to have a component of dimension $\beta(2,d,k) -g$.
When $d= 2g-1-2r$ for a positive integer $r$, this condition becomes 
\begin{equation} \label{eqteix}
g \geq \left\{ \begin{array}{ll}
              \frac{k(k+2r-1)}{2} & \mbox{for} \;k\; \mbox{even}\\
               \frac{(k+1)(k+2r-1)}{2} + 1 & \mbox{for} \; k \; \mbox{odd}.
              \end{array} \right.
\end{equation}
The proof uses degenerations of $C$ and assumes that $C$ and $L$ are both general; however, a semi-continuity argument then shows that the results for non-emptiness are valid for any $C$ and any $L$. Recent work of Osserman \cite{o, o2} contains new information about the dimension of $B(2,L,k)$ and also a non-emptiness result for $k=2$ \cite[Theorem 1.3]{o2}; the dimensional formula has been further extended by Naizhen Zhang \cite{zh}. A complete solution is known for $k \leq 3$ (see \cite{gn} and Remark \ref{rem5}).

In this paper we use a different method to investigate the non-emptiness of $B(2,L,k)$ for $d$ odd. In this case, $M(2,d)$ and $M(2,L)$ are smooth projective varieties. Suppose $d=2g-1-2r$ with $r \geq 1$. Then
$$
\beta(2,d,k) -g  = 3g-3 - k(k+2r-1)
$$
and $B(2,L,k)$ possesses a virtual fundamental class $b(r,k)$ which is independent of the choice of $L$ with $L$ of degree $d$. Note also that, expressed in this form, the expected codimension of $B(2,L,k)$ is $k(k+2r-1)$, which is independent of $g$.
If $b(r,k) \neq 0$, then certainly $B(2,L,k) \neq \emptyset$ for all $L$ of degree $d$. 
Equivalently, the projection $B(2,d,k) \ra \Jac^d(C)$ given by taking determinants is surjective. The 
converse is in general false, since it can (and very often does) happen that $B(2,L,k)$ has dimension $> \beta(2,d,k) -g$. The method is similar to that of \cite{lnp}.

Following some preliminaries in Section 2 concerning the cohomology of $M(2,L)$,we obtain a polynomial formula (independent of $g$) for the class $b(r,k)$ in Section 3. In Section 4, we compute certain values of this polynomial (Proposition \ref{prop3.4}, which depends on a combinatorial lemma (Lemma \ref{lem4.4})). 

As in \cite{lnp}, detailed calculations of $b(r,k)$
are easier if $g$ is a sufficiently large prime. In this way we prove in Section 5,\\

\noindent
{\bf Theorem 5.2}.
{\it Suppose $g$ is a prime with 
$$
g > \max \left\{ \frac{(k+2r-2)(k-1)}{2},\frac{k(k+2r-1)}{3}+1, 2k +2r-1\right\}.
$$
Then $b(r,k) \neq 0$.}\\

\noindent
{\bf Theorem 5.3}.
{\it Let $g_{r,k}$ be the smallest prime such that 
$$
g_{r,k}  > \max \left\{ \frac{(k+2r-2)(k-1)}{2},\frac{k(k+2r-1)}{3}+1, 2k+2r-1 \right\}.
$$
Then, for $r \geq 1$ and $L$ a line bundle of degree $2g-1-2r$,
$$
B(2,L,k) \neq \emptyset \quad \mbox{and} \quad B(2,K \otimes L^*,k+2r-1) \neq \emptyset
$$
for all $g \geq g_{r,k}$.}\\

The condition on $g$ in the statements of the theorems is slightly less restrictive than that of \eqref{eqteix}. Consequently in some cases we have improvements of the results of \cite{te2}.
One can obtain much better results for small values of $r$ and $k$ using Maple. Let $g'_{r,k}$ be the 
smallest prime such that
$$
g'_{r,k} \geq \frac{k(k+2r-1)}{3} +1.
$$ 
Note that this inequality is equivalent to $\beta(2,d,k) -g'_{r,k} \geq 0$. Then we claim that $b(r,k) \neq 0$ for all $g \geq g'_{r,k}$.
The values of $r$ and $k$ for which we have verified this are listed in Remark \ref{rem3.8}.

In Section \ref{r=1}, we calculate $b(1,k)$ exactly for $k\le5$ (using Maple) and consider the possible geometrical interpretation of these calculations. We also obtain precise conditions for non-emptiness of $B(2,L,k)$ for $k\le3$. Finally, in the appendix, we give the proof of Lemma \ref{lem4.4}.

Throughout the paper $C$ is a smooth projective complex curve of genus $g \geq 2$.

We thank the referee for comments leading to some improvements in presentation.

\section{Preliminaries}\label{prelim}

Let $M(2,L)$ and $B(2,L,k)$ be as in the introduction, with $L$ a line bundle of odd degree $d<2g-2$. Write also $d=2g-1-2r$, where $r$ is a positive integer. The moduli space $M(2,L)$ supports a universal bundle $\cE$ on $C \times M(2,L)$ and $B(2,L,k)$ can be viewed as a degeneracy locus in the following way. Choose an effective divisor $D$ of degree $\ge g+r-1$ on $C$. Denote also by $D$ the pullback of $D$ to $C\times M(2,L)$ and consider the exact sequence
$$
0 \ra \cE \ra \cE(D) \ra \cE|_D \ra 0.
$$
Taking direct images via the projection $p_2: C\times M(2,L)\to M(2,L)$, we get the exact sequence
$$
0 \ra {p_2}_*\cE(D) \ra {p_2}_*\cE|_D \ra R^1p_2 \cE \ra 0.
$$
(Note that $p_{2*}\cE = 0$, since $H^0(E) = 0$ for a general $E \in M(2,L)$.)
The Brill-Noether locus $B(2,L,k)$ is then the corank $k$ degeneracy locus of the homomorphism $
{p_2}_*\cE(D) \ra {p_2}_*\cE|_D$. 
Note that the vector bundle $E:=  {p_2}_*\cE(D)$ is of rank $ 2 \deg D+1-2r$ and $F:= {p_2}_*\cE|_D$ of rank $2 \deg D$, so the ``expected dimension'' of $B(2,L,k)$ is 
$$\beta(2,d,k)-g=3g-3-k(k+2r-1).$$
This means that every component of $B(2,L,k)$ has dimension at least $\beta(2,d,k)-g$, but it does not imply that all (or any) of its components are of this dimension or even that it is non-empty when $\beta(2,d,k)-g\ge0$. However it does imply that $B(2,L,k)$ possesses a virtual fundamental class 
$$b(r,k)\in H^{2k(k+2r-1)}(M(2,L),\ZZ).$$ 
Moreover, if $b(r,k)\ne0$, then $B(2,L,k)\ne\emptyset$.

Following \cite{n} and noting that $d=2g-1-2r$, we can  write the Chern classes of $\cE$ as
$$
c_1(\cE) = \alpha + (2g-1-2r) \varphi,
$$
$$
c_2(\cE) = \frac{\alpha^2 - \beta}{4} + \psi + (g-r)\alpha \otimes \varphi.$$
Here $\alpha$ is the positive generator of $H^2(M(2,L),\ZZ)\simeq \ZZ$, 
$$\beta\in H^4(M(2,L),\ZZ),\ \psi\in H^3(M(2,L),\ZZ)\otimes H^1(C,\ZZ)$$ 
and $\varphi$ is the fundamental class of $C$. We define
$\gamma\in H^6(M(2,L),\ZZ)$ by 
$$\psi^2=\gamma\otimes\varphi.$$
The subalgebra of $H^*(M(2,L),\QQ)$ generated by $\alpha$, $\beta$ and $\gamma$ can be written as
$\QQ[\alpha,\beta,\gamma]/I_g$,
and the ideal of relations $I_g$ is explicitly described in \cite{kn}. This ideal depends only on $g$ provided $\deg L$ is odd. 
For any polynomial $f\in\QQ[\alpha,\beta,\gamma]$, we denote by $(f)$ the corresponding cohomology class. It is proved in \cite[Lemma 3.1]{kn} that, if $g\ge g_0$, then
\begin{equation}\label{eqp1}
f\in I_g\Longrightarrow f\in I_{g_0}.
\end{equation} 

In general, it is quite complicated to determine whether a given polynomial $f$ is in $I_g$. However,  Thaddeus \cite{th} gave formulae for the intersection numbers $(\alpha^m\beta^n\gamma^p)$ ($m+2n+3p=3g-3$); we need only a particular deduction from these formulae, which was proved in \cite{lnp}.

\begin{lem}\label{lemp1}
{\em \cite[Lemma 5.1]{lnp}} Suppose that $g$ is an odd prime and $m + 2n + 3p = 3g-3$. Then
$$
(\alpha^m \beta^n \gamma^p) \equiv  \left\{ \begin{array}{ll}
                                         -1 \mod g & \mbox{if}\; p=0 \;\mbox{and} \; m= g-1, 2g-2 \; \mbox{or}\; 3g-3,\\
                                         0 \mod g & \mbox{otherwise}.
                                        \end{array} \right.
$$                                        
\end{lem}

Finally, recall that, if $G$ is any vector bundle of rank $2$ with Chern classes $c_1$, $c_2$, we can write formally
$$1+c_1+c_2=\left(1+ \frac{c_1 - \sqrt{c_1^2 -4c_2}}{2} \right) \cdot\left(1+\frac{c_1+\sqrt{c_1^2 - 4c_2}}{2}\right)$$
and then, for any $n\ge0$, the Chern character of $G$ is given by
\begin{equation}\label{eqp2}
n!\ch_{n}(G) = \left( \frac{c_1 - \sqrt{c_1^2 -4c_2}}{2} \right)^n + \left(\frac{c_1+\sqrt{c_1^2 - 4c_2}}{2}\right)^n.
\end{equation}
We shall write the right hand side of this formula for short in the form
$$\left( \frac{c_1 - \sqrt{c_1^2 -4c_2}}{2} \right)^n + (.)$$
and do the same for other similar expressions.

\section{The fundamental class}\label{fund}

Recall the bundles $E$ and $F$ from Section \ref{prelim} and write $c_i:=c_i(F-E)$. By the Porteous formula \cite[II (4.2)]{acgh}, we have
\begin{equation}\label{eq2.1}
b(r,k)= \left| \begin{array}{cccc}
                                        c_{k+2r-1}&c_{k+2r}&\cdots&c_{2k+2r-2}\\
                                        c_{k+2r-2}&c_{k+2r-1}&\cdots&c_{2k+2r-3}\\
                                        \cdots&&&\cdots\\
                                        c_{2r}&c_{2r+1}&\cdots&c_{k+2r-1}
                                        \end{array}  \right|.
\end{equation}

Our main object in this section is to compute the Chern classes $c_i$. For this, note first that, if we choose $D=q_1+\ldots +q_{\deg D}$ with distinct points $q_i$, then $F\simeq\oplus_i^{\deg D}\cE|_{\{q_i\}\times M(2,L)}$. Topologically the bundles $\cE|_{\{q_i\}\times M(2,L)}$ are all isomorphic and we denote any one of them by $\cE_M$. We have then
\begin{equation}\label{eq2.2}
\ch(F)=\deg(D)\ch(\cE_M).
\end{equation}

\begin{lem} \label{lem2.1} For $n\ge1$,
$$
2^{n+1} \ch_{n}(F - E) =
$$
$$
\frac{1}{n!}(\alpha - \sqrt{\beta})^n(2r-1)
- \frac{1}{n!}(\alpha - \sqrt{\beta})^n\left(\frac{\alpha}{\sqrt{\beta}} + \frac{2\gamma}{\beta^{\frac{3}{2}}} \right) 
- \frac{1}{2(n-1)!}(\alpha - \sqrt{\beta})^{n-1} \frac{4\gamma}{\beta} + (.).
$$
\end{lem}

\begin{proof}
Using \eqref{eq2.2} and Grothendieck-Riemann-Roch, we obtain
\begin{eqnarray*}
\ch(F-E)& =& \deg(D) \ch(\cE_M) - \deg(D) \ch(\cE_M)  - \ch(\cE(C)) + (g-1)\ch(\cE_M)\\
&=& (g-1)\ch(\cE_M) - \ch(\cE(C)),
\end{eqnarray*}
where 
$$
\ch \cE(C) = \int_C \ch \cE.
$$
 
Now, by \eqref{eqp2},
$$
n! \ch_n \cE_M =\left(\frac{\alpha - \sqrt{\beta}}{2} \right)^n + \left( \frac{\alpha + \sqrt{\beta}}{2} \right)^n
$$
and
$$
(n+1)! \ch_{n+1}\cE = \left( \frac{\alpha + (2g-1-2r)\varphi - \sqrt{\beta -4\psi -2\alpha \varphi}}{2} \right)^{n+1} + (.).
$$
Now expand
\begin{eqnarray*}
\sqrt{\beta - 4 \psi - 2 \alpha \varphi} & = & \sqrt{\beta} \left( 1 - \frac{4\psi}{\beta} - \frac{2 \alpha \varphi}{\beta} \right)^{\frac{1}{2}}\\
& = & \sqrt{\beta} \left( 1 - \left(\frac{2\psi}{\beta} + \frac{\alpha\varphi}{\beta} \right) - 2 \frac{\gamma \varphi}{\beta^2} \right).
\end{eqnarray*}
Hence
$$
(n+1)! \ch_{n} \cE(C) = \frac{1}{2^{n+1}} \int_C  \left( \alpha - \sqrt{\beta} + \frac{2\psi}{\sqrt{\beta}} + \left(2g-1-2r + \frac{\alpha}{\sqrt{\beta}} + 
\frac{2\gamma}{\beta^{\frac{3}{2}}} \right) \varphi \right)^{n+1} + \left(.\right)
$$
$$
= \frac{1}{2^{n+1}} \left[ (n+1)(\alpha - \sqrt{\beta})^n \left(2g-1-2r +\frac{\alpha}{\sqrt{\beta}} + \frac{2\gamma}{\beta^{\frac{3}{2}}} \right) +
{n+1 \choose 2}(\alpha - \sqrt{\beta})^{n-1} \frac{4\gamma}{\beta} \right] + (.).
$$ 
Since
$$
 \ch_n(F-E) = (g-1) \ch_n \cE_M -\ch_n \cE(C),
$$
this implies the assertion.
\end{proof}

Turning now to Chern classes, we have of course $c_0=1$. We write also $c_n=0$ for $n<0$.

\begin{prop}\label{prop2.2}
Let $c_i=c_i(F-E)$. Then, for every integer $n$,
$$
(n+4)c_{n+4} + (2n +6 -r)\alpha c_{n+3} 
+ \left[ \left(n+2-r\right) \alpha^2 + (2n +5-2r) \frac{\alpha^2-\beta}{4} \right] c_{n+2}
$$
$$
+\left[(2n+3-3r) \alpha \frac{\alpha^2 - \beta}{4} + \frac{\gamma}{2} \right]c_{n+1} 
+ \frac{1}{16}(\alpha^2-\beta)^2(n+1-2r)c_{n} = 0.
$$
\end{prop}

\begin{proof}
Let $c(t)=\sum_0^\infty c_nt^n$. Consider
\begin{eqnarray*}
c(t) &=& \exp(\log(c(t))\\& =& \exp\left( \ch_1(F-E) t - \ch_2(F-E) t^2 + \cdots + (-1)^{n-1} (n-1)!\ch_n(F-E) t^n +\cdots \right)
\end{eqnarray*}
$$
= \exp \left[ (2r-1) \sum_{n=1}^\infty \frac{(-1)^{n-1}}{2^{n+1}} (\alpha - \sqrt{\beta})^n \frac{t^n}{n} \right. 
$$
$$
 \left. - \sum_{n=1}^\infty \frac{(-1)^{n-1}}{2^{n+1}} (\alpha - \sqrt{\beta})^n \left( \frac{\alpha}{\sqrt{\beta}} 
+ \frac{2\gamma}{\beta^{\frac{3}{2}}} \right) \frac{t^n}{n} 
- \sum_{n=1}^\infty \frac{(-1)^{n-1}}{2^n} (\alpha - \sqrt{\beta})^{n-1} \frac{\gamma}{\beta} t^n + (.) \right].
$$
So 
$$
\frac{d}{dt}(c(t)) = c(t) \left[ (2r-1) \sum_{n=1}^\infty \frac{(-1)^{n-1}}{2^{n+1}}(\alpha - \sqrt{\beta})^n t^{n-1} \right. 
$$
$$
\left. - \sum_{n=1}^\infty \frac{(-1)^{n-1}}{2^{n+1}} (\alpha - \sqrt{\beta})^n 
\left( \frac{\alpha}{\sqrt{\beta}} + \frac{2\gamma}{\beta^{\frac{3}{2}}} \right)t^{n-1} \right.
$$
$$
\left. - \sum_{n=1}^\infty \frac{(-1)^{n-1}}{2^{n}}  (\alpha - \sqrt{\beta})^{n-1} \frac{\gamma}{\beta} n t^{n-1} + (.) \right]
$$
$$
= c(t) \left[ \frac{2r-1}{4} (\alpha -\sqrt{\beta})\frac{1}{1 +\frac{\alpha - \sqrt{\beta}}{2} t} 
-  \frac{\frac{1}{4}(\alpha - \sqrt{\beta})  \left( \frac{\alpha}{\sqrt{\beta}} +
\frac{2\gamma}{\beta^{\frac{3}{2}}}\right)}{1 + \frac{\alpha - \sqrt{\beta}}{2}t} \right.
$$
$$
\left. - \frac{ \gamma}{2\beta} \frac{1}{(1 + \frac{\alpha - \sqrt{\beta}}{2}t)^2} + (.) \right].
$$

Substituting $c(t)=\sum_0^\infty c_nt^n$, multiplying by 
$$\left(1+\frac{\alpha-\sqrt\beta}2 t\right)^2\cdot\left(1+\frac{\alpha+\sqrt\beta}2 t\right)^2$$ 
and comparing the coefficients of $t^{n+3}$ gives the result (after some algebraic manipulation).
\end{proof}

Proposition \ref{prop2.2} allows us to consider $c_i(F-E)$ as a polynomial $c_i(\alpha,\beta,\gamma)$. We can therefore define a polynomial
$$
P_k(\alpha,\beta,\gamma) :=  \left| \begin{array}{cccc}
                                        c_{k+2r-1}&c_{k+2r}&\cdots&c_{2k+2r-2}\\
                                        c_{k+2r-2}&c_{k+2r-1}&\cdots&c_{2k+2r-3}\\
                                        \cdots&&&\cdots\\
                                        c_{2r}&c_{2r+1}&\cdots&c_{k+2r-1}
                                        \end{array}  \right|
$$                                        
such that
$$(P_k(\alpha,\beta,\gamma))=b(r,k).$$
\begin{rem}\label{rem33} {\rm Note that here $\alpha,\beta,\gamma$ are indeterminates (not cohomology classes) of degree $2,4,6$ respectively, making $P_k$ a homogeneous polynomial of degree $2k(k+2r-1)$.}
\end{rem}

In the next proposition, we obtain a simpler recurrence relation for $c_i(\alpha,\beta,0)$.

\begin{prop}
Let $c_i=c_i(\alpha,\beta,0)$. Then, for every integer $n$, 
\begin{equation}\label{erecurr}
(n+2)c_{n+2} + \left(n+1-r \right) \alpha c_{n+1} + (n+1-2r)\frac{\alpha^2 - \beta}{4}c_{n} = 0.
\end{equation}
\end{prop}

\begin{proof}
By definition, $c_0 = 1$. From Lemma \ref{lem2.1} we get $c_1 = r \alpha$. 
If $\gamma = 0$, we have the equation
$$
\left(\left(1+ \frac{\alpha t}{2}\right)^2 - \frac{\beta}{4} t^2 \right) \sum_{n=1}^\infty nc_n t^{n-1} 
$$
$$
= \left[\frac{1}{4} \left(1+ \frac{\alpha + \sqrt{\beta}}{2}t \right) (\alpha - \sqrt{\beta}) \left( 2r-1 - \frac{\alpha}{\sqrt{\beta}} \right) 
+ (.) \right] \sum_{n=0}^\infty c_n t^n.
$$
Comparing the coefficients of $t^{n+1}$ gives \eqref{erecurr}.
\end{proof}

\begin{cor}  \label{cor2.4}
  If $\beta = \alpha^2$, then $c_n = 0$ for $n \geq r+1$.
\end{cor}

\begin{proof}
This follows immediately from \eqref{erecurr}. 
\end{proof}

\section{Computation of $P_k$}\label{comp}

Our object in this section is to compute $P_k(1,\beta,0)$ up to a non-zero constant. This polynomial contains a term $c\beta^{\frac{k(k+2r-1)}2}$ for some constant $c$; by Remark \ref{rem33}, this term is just $P_k(0,\beta,0)$. We prove first that $c\ne0$ by showing that $P_k(0,4,0)\ne0$.
 
For this, consider $\widetilde c_i := c_i(0,\beta,0)$. The recurrence relation for the $\widetilde c_i$ is 
$$
\widetilde c_0 = 1, \; \widetilde c_1 = 0\quad \mbox{and} \quad (n+2)\widetilde c_{n+2} = \frac{\beta}{4}(n+1-2r) \widetilde c_n
$$
for $n \geq 0$.

\begin{lem} \label{lem3.1}
For all $n$ we have $\widetilde c_{2n+1} = 0$ and, for $n \geq r$,
$$
\widetilde c_{2n} = (-1)^r \frac{(2r-1)(2r-3)\cdots 1}{2^{2n-r}n(n-1) \cdots (n-r+1)} \cdot \frac{(2n-2r)!}{((n-r)!)^2} \left(\frac{\beta}{4} \right)^n.
$$Furthermore, if $\beta = 4$, then for any odd prime $p > \max\{2r-1,n\}$,
$$
\widetilde c_{2n} \equiv (-1)^n  e_{n} \mod p,
$$
where $e_n$ is defined by $(1 + t)^{\frac{p+2r-1}{2}} = \sum_{i=0}^{\frac{p+2r-1}{2}} e_i t^i$. 
\end{lem}

\begin{proof}
The fact that $\widetilde c_{2n+1} = 0$ follows directly from the recurrence relation. For $n \geq r$ we can solve the recurrence relation for $\widetilde c_{2n}$
giving
\begin{eqnarray*}
\widetilde c_{2n} &=& \frac{2n-1-2r}{2n} \cdot \frac{2n-3-2r}{2n-2} \cdots \frac{1-2r}{2} \left( \frac{\beta}{4} \right)^n\\
&=& (-1)^r (2r-1)(2r-3) \cdots 1 \cdot \frac{(2n-2r)!}{2^{2n-r} n! (n-r)!}\left(\frac{\beta}{4} \right)^n.
\end{eqnarray*}
This gives the second assertion.

If $\beta = 4$, as in the proof of \cite[Lemma 4.4]{lnp} we see that 
$$
{\frac{p-1}{2} \choose n} \equiv (-1)^n \frac{(2n)!}{2^{2n}(n!)^2} \mod p
$$
for $p$ an odd prime, $p>n$. So, for $p>\max\{2r-1,n\}$,
\begin{eqnarray*}
\widetilde c_{2n} & \equiv & (-1)^n \frac{(2r-1)(2r-3)\cdots 1}{2^{r}n(n-1) \cdots (n-r+1)} {\frac{p-1}{2} \choose n-r} \\
& \equiv & (-1)^n \frac{(2r-1)(2r-3) \cdots 1}{(p-1+2r)(p-3+2r) \cdots (p+1)} {\frac{p+2r-1}{2} \choose n}\\
& \equiv & (-1)^n {\frac{p+2r-1}{2} \choose n} \mod p,
\end{eqnarray*}

giving the last assertion.
\end{proof}

\begin{lem} \label{lem3.2}
Suppose $\beta = 4$. For integers $u \geq v \geq r$, let
$$
A_{u,v} = \left( \begin{array}{cccc}
             \widetilde c_{2u}& \widetilde c_{2u+2}&\cdots &\widetilde c_{4u-2v}\\
             \widetilde c_{2u-2}&\widetilde c_{2u}& \cdots&\widetilde c_{4u-2v-2}\\
             \cdots&&&\cdots\\
             \widetilde c_{2v}& \widetilde c_{2v+2}& \cdots& \widetilde c_{2u}
             \end{array} \right).         
$$             
Then, for any odd prime $p > \max \{ 2u-v,2r+2u-2v-1 \} $,
$$
\det A_{u,v} \not \equiv 0 \mod p.
$$
\end{lem}

\begin{proof} Lemma \ref{lem3.1} gives
\begin{eqnarray*}
\det A_{u,v} & \equiv & (-1)^{\delta} \left| \begin{array}{cccc}
                               e_u&e_{u+1}& \cdots & e_{2u-v}\\
                               e_{u-1}& e_u & \cdots & e_{2u-v-1}\\
                               \cdots&&&\cdots\\
                               e_v&e_{v+1}&\cdots& e_u
                               \end{array} \right|  \mod p,                                
\end{eqnarray*}
where 
$$
\delta = \left\{ \begin{array}{ll}
                  -1 & \; \mbox{if} \; u \; \mbox{and} \; v \; \mbox{are both odd},\\
                  +1 & \; \mbox{otherwise}.
                 \end{array} \right.
$$ 
So, by \cite[equation (A.6)]{fh},
$$
  \det A_{u,v} \equiv (-1)^{\delta} S_{u-v+1, \dots, u-v+1,0,\dots,0} (1,\dots,1) \mod p,
$$
where $u-v+1$ is repeated $u$ times, $0$ is repeated $\frac{p+2r-1}{2} -u$ times and $S$ is the Schur polynomial.
Using \cite[Exercise A.30]{fh} we see that for $p > 2r+2u-2v-1$,
$$
S_{u-v+1,\dots,u-v+1,0\dots,0} (1,\dots,1) \not \equiv 0 \mod p. 
$$
This implies the assertion. 
\end{proof}

\begin{prop} \label{prop3.3}
$P_k(0,4,0) \not \equiv 0 \mod p$ for any odd prime $p > k+2r-2$. 
\end{prop}

\begin{proof}
By Lemma \ref{lem3.1}, if $k$ is odd,
\begin{eqnarray*}
P_k(0,4,0) & = & \left| \begin{array}{ccccccc}
                        \widetilde c_{k+2r-1}&0&\widetilde c_{k+2r+1}&0&\cdots&0&\widetilde c_{2k+2r-2}\\ 
                        0&\widetilde c_{k+2r-1}&0&\widetilde c_{k+2r+1}&\cdots&\widetilde c_{2k+2r-4}&0\\
                        \cdots&\cdots&&&&\cdots&\cdots\\
                        \widetilde c_{2r}& 0 &\widetilde c_{2r+2}& 0& \cdots& 0&\widetilde c_{k+2r-1}
                        \end{array}  \right|\\
& = & \left| \begin{array}{cc}
             A_{\frac{k+2r-1}{2},r}& 0\\
             0& A_{\frac{k+2r-1}{2},r+1}
             \end{array} \right|
\end{eqnarray*}
by permutations of rows and columns.
Similarly for $k$ even,
$$
P_k(0,4,0) = \left| \begin{array}{cc}
                    A_{\frac{k+2r}{2},r+1}&0\\
                    0& A_{\frac{k+2r-2}{2},r}
                    \end{array} \right|.
$$                    
The assertion in both cases follows from Lemma \ref{lem3.2}.
\end{proof}

We turn now to a consideration of $P_k(1,\beta,0)$. We begin with a lemma, which will be proved in the appendix.

\begin{lem} \label{lem4.4}
$$c_{2r}(1,\beta,0)=\frac1{2^{2r}(2r)!}\prod_{i=1}^r(1-(2i-1)^2\beta).$$
\end{lem}

\begin{prop} \label{prop3.4} For some non-zero constant $c$,
\begin{equation} \label{form}
P_k(1,\beta,0) = c \cdot \prod_{i=1}^r \left(\beta - \frac{1}{(2i-1)^2} \right)^k \cdot \prod_{i=1}^{k-1} \left( \beta - \frac{1}{(2r + 2i-1)^2} \right)^{k-i}.
\end{equation}
\end{prop}

\begin{proof}
Note first that if $c_{2r}(1,\beta,0) = 0$, then by \eqref{erecurr}, $c_n(1,\beta,0) = 0$ for all $n \geq 2r$. So by Lemma \ref{lem4.4} the matrix defining $P_k(1,\beta,0)$ 
is the zero matrix for $\beta = \frac{1}{(2i-1)^2}$ for $1 \leq i \leq r$. This gives the first product in formula \eqref{form}.

Let $\ell$ be an integer $\geq r+1$. Consider a sequence of numbers $d_n$ defined for $n \geq 2r-1$ and satisfying the recurrence relation
$$
(n+2)d_{n+2} + (n+1-r)d_{n+1} + (n+1-2r) \frac{\ell(\ell-1)}{(2\ell -1)^2}d_n = 0
$$
for $n \geq 2r-1$. Note that, for any value of $d_{2r}$, there is a unique solution for $d_n$ for $n \geq 2r$. We claim that for $n \geq 2r$,
$$
d_n = s\left(- \frac{\ell -1}{2\ell-1} \right)^n (a_0 + a_1n + \cdots + a_{\ell -r-1}n^{\ell-r-1})
$$
for some constants $s,a_0,\dots,a_{\ell-r-1}$ with $a_0, \dots,a_{\ell-r-1}$ not all zero.

We need to show that there exist constants $a_0,\dots,a_{\ell-r-1}$, not all zero, such that
$$
\left(- \frac{\ell-1}{2\ell-1} \right)^{n+2} (n+2)(a_0 + a_1(n+2) + \cdots + a_{\ell-r-1}(n+2)^{\ell-r-1}) 
$$
$$
+\left( - \frac{\ell-1}{2\ell-1} \right)^{n+1} (n+1-r) (a_0 + a_1(n+1) + \cdots + a_{\ell-r-1}(n+1)^{\ell-r-1}) 
$$
$$
+ \left( - \frac{\ell-1}{2\ell-1} \right)^n  \frac{\ell (\ell-1)}{(2\ell -1)^2} (n+1-2r)(a_0 + a_1n + \cdots + a_{\ell-r-1} n^{\ell-r-1} ) =  0
$$
for all $n$, i.e.
$$
(n+2)(a_0 + a_1(n+2) + \cdots + a_{\ell-r-1}(n+2)^{\ell-r-1})
$$
$$
- \frac{2\ell-1}{\ell-1}  (n+1-r) (a_0 + a_1(n+1) + \cdots + a_{\ell-r-1}(n+1)^{\ell-r-1}) 
$$
$$
+ \frac{\ell}{\ell-1} (n+1-2r)(a_0 + a_1n + \cdots + a_{\ell-r-1} n^{\ell-r-1} )  =  0.
$$
One checks that the coefficients of $n^{\ell-r}$ and $n^{\ell-r-1}$ are both zero. This leaves us with $\ell -r-1$ homogeneous linear equations in $a_0, \dots, a_{\ell-r-1}$
which have a non-trivial solution. The claim follows.
Note that, if $d_{2r} = 0$, then $d_n = 0$ for all $n \geq 2r$, which is impossible unless $s = 0$.

According to \eqref{erecurr}, $c_n = c_n\left(1, \frac{1}{(2\ell -1)^2},0\right)$ satisfies the recurrence relation
$$
(n+2)c_{n+2} + (n+1-r)c_{n+1} + (n+1-2r) \frac{\ell(\ell-1)}{(2\ell -1)^2} c_n = 0.
$$
Now choose $s$ such that $c_{2r} = d_{2r}$. Then $c_n = d_n$ for all $n \geq 2r$. It follows that the rows of the matrix defining $P_k(1,\frac{1}{(2\ell-1)^2},0)$ 
lie in a $\QQ$-vector space 
of dimension $\leq \ell-r$. So $\frac{1}{(2\ell -1)^2}$ is a zero of multiplicity at least $k - \ell + r$ of the polynomial 
$P_k(1,\beta,0)$. This gives the second product in formula \eqref{form}.
Since the degree of $P_k(1,\beta,0)$ is $\frac{k(k+2r-1)}{2}$ by Remark \ref{rem33} and Proposition \ref{prop3.3},
this completes the proof of the proposition.
\end{proof}

\section{Main Theorem}\label{main}

For $g \geq 2k +2r -1$, define
$$
w:= ((g-1)!2^{g-1})^k P_k(\alpha,\beta,\gamma) \in \ZZ[\alpha,\beta,\gamma]
$$
and write 
$$
w = \sum_{j\geq0} M_j \beta^j \alpha^{k(k+2r-1) -2j} + \gamma R(\alpha, \beta, \gamma)
$$
with $M_j \in \ZZ$. Then, writing
$$
e := 3g-3-k(k+2r-1),
$$
we define
$$
w_0 := \alpha^{e} w = \sum_{j\geq 0} M_j \beta^j \alpha^{3g-3-2j} + \gamma \widetilde R(\alpha,\beta,\gamma).
$$
and 
$$
w_{\ell} := \alpha^{e- 2\ell} \beta^{\ell} w,
$$
for $1 \leq \ell \leq \frac{e}{2}$.
If $g$ is a prime, then, according to Lemma \ref{lemp1},
\begin{equation} \label{eq3.1}
(w_0) \equiv -M_0 - M_{\frac{g-1}{2}} - M_{g-1} \mod g
\end{equation}
and 
\begin{equation}  \label{eq3.2}
(w_{\ell}) \equiv - M_{\frac{g-1}{2} - \ell} - M_{g-1-\ell} \mod g. 
\end{equation}
Note that, if $(w_0)\not\equiv0\mod g$ or $(w_\ell)\not\equiv0\mod g$, then $b(r,k)\ne0$.

Define as in \cite[Section 5]{lnp}, for $0 \leq i < \frac{g-1}{2}$,
$$
M'_i :\equiv M_i + M_{i + \frac{g-1}{2}} + M_{i + g-1} \mod g
$$
with $0 \leq M'_i \leq g-1$ and consider
$$
q(\beta) := M'_0 + M'_1 \beta + \cdots + M'_{\frac{g-3}{2}} \beta^{\frac{g-3}{2}} \in \FF_g[\beta].
$$

\begin{lem} \label{lem3.5}
Suppose $g$ is a prime. If 
$$g > \max \left\{\frac{k(k+2r-1)}{3}+1, 2k+2r-1 \right\},$$ 
then $q(\beta)$ is not identically zero. Moreover, $q$ has $k+r-1$ distinct zeros different from $0$.
\end{lem}

\begin{proof}
 Let $x \in \ZZ, \; 1 \leq x \leq g-1$. Using the fact that $x^{g-1} \equiv 1 \mod g$, we see that 
\begin{equation} \label{eq3.3}
P_k(1,x^2,0) \equiv q(x^2) \mod g,
\end{equation}
since $M_i = 0$ for $i \geq \frac{3g-3}{2}$. This is true, since $\frac{3g-3}{2} > \frac{k(k+2r-1)}{2}$ by hypothesis and the degree of $P_k(1,\beta,0)$ 
as a polynomial in $\beta$ is $\frac{k(k+2r-1)}{2}$.

By Proposition \ref{prop3.4}, $P_k(1,x^2,0)$ has precisely $k+r-1$ distinct zeros. The field $\FF_g$ contains $\frac{g-1}{2}$ non-zero squares. Since 
$k +r -1< \frac{g-1}{2}$ by hypothesis, there exists an integer $x,\; 0 < x < g-1$, such that 
$$
P_k(1,x^2,0) \not \equiv 0 \mod g.
$$
Both assertions now follow from \eqref{eq3.3}.
\end{proof}

\begin{theorem} \label{thm3.6}
Suppose $g$ is a prime with
$$
g > \max \left\{ \frac{(k+2r-2)(k-1)}{2},\frac{k(k+2r-1)}{3}+1, 2k +2r-1\right\}.
$$
Then $b(r,k) \neq 0$.
\end{theorem}

\begin{proof}
If $M'_0 \neq 0$, then $b(r,k) \neq 0$ by \eqref{eq3.1}. If $M'_0 = 0$, then $M'_{k_0} \neq 0$ for some $k_0 \geq k+r$ by Lemma \ref{lem3.5}. 
We have $k_0 < \frac{g-1}{2}$ and we claim that 
$$
\frac{g-1}{2} - k_0 \leq \frac{e}{2}.
$$
In fact, this is equivalent to $g-1-2k_0 \leq 3g-3 - k(k+2r-1)$ which is true if $g \geq \frac{(k+2r-2)(k-1)}{2}$.
The last inequality is true by hypothesis.

So consider $w_{\ell}$ with $\ell = \frac{g-1}{2} - k_0$. Note that 
$$
M_{k_0} + M_{\frac{g-1}{2} + k_0} \equiv M'_{k_0} \mod g,
$$
provided that $M_{g-1+k_0} \equiv 0 \mod g$. This is true if $g-1+k_0 > \frac{k(k+2r-1)}{2}$ which holds by hypothesis, since $k_0 \geq k+r$. So $b(r,k) \neq 0$ by \eqref{eq3.2}.
\end{proof}

\begin{theorem} \label{thm3.7}
Let $g_{r,k}$ be the smallest prime such that 
$$
g_{r,k}  > \max \left\{ \frac{(k+2r-2)(k-1)}{2},\frac{k(k+2r-1)}{3}+1, 2k+2r-1 \right\}.
$$
Then, for $r \geq 1$ and $L$ a line bundle of degree $2g-1-2r$,
$$
B(2,L,k) \neq \emptyset \quad \mbox{and} \quad B(2,K \otimes L^*,k+2r-1) \neq \emptyset
$$
for all $g \geq g_{r,k}$.
\end{theorem}

\begin{proof} For $B(2,L,k)$ this follows from Theorem \ref{thm3.6} and \eqref{eqp1}. The last part of the assertion follows from Serre duality.
\end{proof}

\begin{rem} \label{rem5}
{\rm For $k\le 3$ (also for $k=4, r\le3$ and for $k=5,r=1$), the third term in the maximum of Theorems \ref{thm3.6} and \ref{thm3.7} is strictly greater than the first term. In fact, for $k\le3$, precise conditions for the non-emptiness of $B(2,L,k)$ are known and provide improvements on the results of Theorem \ref{thm3.7}. For further details on all these cases, see Section \ref{r=1}.}
 
\end{rem}

\begin{rem} \label{rem3.8}
{\rm For $k \geq 4$, we can improve the results of Theorem \ref{thm3.7} using Maple. Note that the definitions of $w_0$ and $w_{\ell}$ require only that $g$ be a prime number with $g \ge 2k+2r-1$ and $\ell \geq 1$. 
Let $g'_{r,k}$ be the smallest prime such that 
\begin{equation}   \label{eq5.4} 
g'_{r,k} \geq \frac{k(k+2r-1)}{3} +1.
\end{equation}
For $k \geq 4$, we have $\frac{k(k+2r-1)}{3} + 1 \geq 2k+2r-1$ except when $k=4$ and $r = 1$ or $2$. In these cases we find that \eqref{eq5.4} 
implies that $g'_{r,k} \geq 2k +2r -1$. So this holds always. 

Suppose we can prove directly that \eqref{eq3.1} or \eqref{eq3.2} gives an integer which is not congruent to $0$ modulo $g'_{r,k}$.
Then it follows by \eqref{eqp1} that $b(r,k) \neq 0$ and $B(2,L,k) \neq \emptyset$ for all $g \geq g'_{r,k}$ and for every line bundle $L$ on $C$ of degree $2g -2r-1$. 
We carried this out for 
$$
r=1,\; 4 \leq k \leq 17 \quad \mbox{and} \quad 2 \leq r \leq 5,\; 4 \leq k \leq 10.
$$ 
For $(r,k) =
(1,5),(1,9),(1,12),(1,13),(1,14),(1,17),(2,4),(2,6),(2,8),$ $(2,9),$ $(3,4),$ $(3,6),$ $(3,7),$ $(3,9),$ $(4,8),$(5,4),$ (5,6)$ and $(5,8)$, 
this gives the best possible 
result for $b(r,k) \neq 0$, namely that $b(r,k)\ne0$ whenever $\beta(2,d,k) -g \ge0$.
}
\end{rem}

\section{Further results for small $k$}\label{r=1}
Ideally, we would like to prove that $b(r,k)\ne0$ whenever
$$g\ge g^0_{r,k}:=\left\lceil\frac{k(k+2r-1)}3\right\rceil+1,$$
since this is equivalent to $\beta(2,d,k)-g\ge0$. Recall that by \eqref{eqp1}, it is sufficient to do this for $g = g^0_{r,k}$.
If $g^0_{r,k}$ is prime, we have $g^0_{r,k}=g'_{r,k}$ and the methods of Section \ref{main} apply. Otherwise, the calculations become much more complicated. However, a complete calculation of the cohomology class $b(r,k)$ would be of interest not only for proving that $b(r,k)\ne0$ but for investigating the geometry of the Brill-Noether locus. With the help of Maple, using \eqref{eq2.1}, Proposition \ref{prop2.2} and Thaddeus' formulae for the intersection numbers \cite{th}, we have carried out the computation for $r=1$ (that is, $d=2g-3$) and $k\le5$ in the case $g = g^0_{r,k}$. For $k\le3$ (and partially for $k=4$), we can interpret these results geometrically. We include some more precise information on non-emptiness for $k\le3$ and arbitrary $r$.

\begin{ex}\label{ex1} {\rm Let $k=1$. When $r=1$, we have $g^0_{1,1}=2$ and $d=1$. In this case, it is easy to see by hand that $P_1(\alpha,\beta,\gamma)=\frac18(\alpha^2-\beta)$ and that the intersection number  $(\alpha\cdot P_1(\alpha,\beta,\gamma))$ is $1$. Geometrically, it is well known that $M(2,L)$ is a smooth intersection of quadrics in ${\mathbb P}^5$ and that $B(2,L,1)$ is a line contained in this intersection \cite[Theorem 2]{n2}. The elements of $B(2,L,1)$ are precisely the non-trivial extensions
\begin{equation}\label{eqb}
0\lra{\mathcal O}\lra E\lra L\lra0.
\end{equation}
This works for all $L$ of degree $1$.

More generally, we know that $B(2,L,1) \neq \emptyset$ if and only if $d:=\deg L \geq 1$. This is independent of $L$ and does not even require $d$ to be odd. In fact, it is obvious that $B(2,L,1)=\emptyset$ for $d\le0$, while it is well known (and easy to see) that, if $d\ge1$, the general extension \eqref{eqb} is stable.  If $d\ge2g-1$, then $B(2,L,1)=M(2,L)$. On the other hand, if $1\le d\le 2g-2$, it is known that $B(2,d,1)$ has dimension $\beta(2,d,1)$ \cite[Theorem III.2.4]{sun}. It follows that, at least for general $L$ of degree $d$ in this range, $B(2,L,1)$ has dimension $\beta(2,d,1)-g$. (This does not follow directly from \cite[Theorem 3.3 and Corollary 3.5]{gn}, which imply only that $B(2,L,1)$ has an irreducible component of this dimension.)

When $d=2g-1-2r$, the condition $d\ge1$ is equivalent to $g \geq r+1$ (note that Theorem \ref{thm3.7} requires $g\ge2r+2$), thus implying that $b(r,1)=0$ when $r=g$ (this corresponds to the case $d=-1$), while $b(r,1)\ne0$ for $r\le g-1$. On the other hand, if $r=g\ge3$, we have $\beta(2,-1,1)-g\ge0$, so $b(r,1)=0$ is a non-trivial relation in the cohomology of $M(2,L)$. Equivalently, $c_{2r}$ belongs to the ideal $I_r$ (see \eqref{eq2.1} and the discussion preceding \eqref{eqp1}); in fact, $c_{2r}\in I_r\setminus I_{r+1}$. Now, a recursive formula for generators of $I_r$ is known \cite{kn}, the generators being denoted there by $\zeta_r$, $\zeta_{r+1}$ and $\zeta_{r+2}$. For $r=3$, one can easily simplify these generators to give 
\begin{equation}\label{eqb2}
\zeta_3= \alpha^3+5\alpha\beta+4\gamma,\ \eta_4:=\alpha^4+2\alpha^2\beta-3\beta^2,\ \eta_5:=2\alpha^5+7\alpha^3\beta.
\end{equation}
(The generators are in fact implicit in \cite[Theorem 4]{ram}, which gives a complete description of the cohomology ring of $M(2,L)$ when $g=3$.) Now, from Proposition \ref{prop2.2},
$$c_{6}=\frac1{46080}(\alpha^6-35\alpha^4\beta+259\alpha^2\beta^2-225\beta^3-160\alpha^3\gamma+928\alpha\beta\gamma+640\gamma^2).$$
Combining this with \eqref{eqb2}, we obtain
$$46080c_6=(-80\alpha^3+32\alpha\beta+160\gamma)\zeta_3+(17\alpha^2+75\beta)\eta_4+32\alpha\eta_5.$$
In fact, this expression is unique. Similarly, for $r=4$, we have 
\begin{eqnarray*}
\zeta_4&=&\alpha^4+14\alpha^2\beta+9\beta^2+16\alpha\gamma\\
\zeta_5&=&\alpha^5+30\alpha^3\beta+89\alpha\beta^2+40\alpha^2\gamma+88\beta\gamma\\
\zeta_6&=&\alpha^6+55\alpha^4\beta+439\alpha^2\beta^2+225\beta^3+80\alpha^3\gamma+688\alpha\beta\gamma+160\gamma^2
\end{eqnarray*}
and
\begin{eqnarray*}
c_8&=&\frac1{2^88!}(\alpha^8-84\alpha^6\beta+1974\alpha^4\beta^2-12916\alpha^2\beta^3+11025\beta^4\\
&&-448\alpha^5\gamma+11648\alpha^3\beta\gamma-48064\alpha\beta^2\gamma+17920\alpha^2\gamma^2-39424\beta\gamma^2).
\end{eqnarray*}
From this, we obtain the unique expression
\begin{eqnarray*}
2^88!c_8&=&(70\alpha^4+1820\alpha^2\beta+3150\beta^2+2020\alpha\gamma)\zeta_4\\
&&-(56\alpha^3+412\alpha\beta+308\gamma)\zeta_5-(13\alpha^2+77\beta)\zeta_6.
\end{eqnarray*}
 It would certainly be possible to obtain similar formulae for higher values of $r$ using Maple, which might enable one to see a pattern which would allow one to guess a general formula.}
\end{ex}

\begin{ex}\label{ex2} {\rm Let $k=2$. When $r=1$, we have $g^0_{1,2}=3$ and $d=3$. This time $b(1,2)$ is itself a top dimensional class and is numerically equal to $1$ (this can again be done by hand or using Maple), so in particular there exists $E\in B(2,L,2)$. Certainly $E$ has no line subbundle of degree $\ge2$ and hence no line subbundle with $h^0\ge2$. Hence $E$ is generically generated and there is an exact sequence
\begin{equation}\label{eqb3}0\lra{\mathcal O}^2\lra E\lra T\lra0,
\end{equation}
where $T$ is a torsion sheaf. 
Suppose that $h^0(L)=1$ and that $L\simeq {\mathcal O}(p+q+r)$ with $p$, $q$, $r$ all different. Then $T\simeq{\mathcal O}_p\oplus{\mathcal O}_q\oplus{\mathcal O}_r$ and extensions of $T$ by ${\mathcal O}^2$ are classified (up to automorphisms of $T$) by points $(x,y,z)\in{\mathbb P}^1\times{\mathbb P}^1\times{\mathbb P}^1$. Stability of $E$ implies that $x$, $y$, $z$ are all distinct (for example, if $x=y$, then ${\mathcal O}(p+q)$ is a subbundle of $E$). There is just one orbit of points of this type for the action of $\mbox{Aut}({\mathcal O}^2)=\mbox{GL}(2,{\mathbb C})$, so $E$ is uniquely determined.

Another way of proving that $B(2,L,2)$ consists of just one point is to look at extensions
$$
0\lra{\mathcal O}(p)\lra E\lra{\mathcal O}(q+r)\lra0.
$$
Any non-trivial extension defines a stable bundle $E$. Moreover $h^0(E)=2$ if and only if the element classifying the extension belongs to 
$$\Ker: H^1({\mathcal O}(p-q-r))\lra\Hom(H^0({\mathcal O}(q+r)), H^1({\mathcal O}(p))).$$
It is easy to show that this kernel has dimension $1$. Replacing $p$ by $q$ or $r$ could conceivably give up to $3$ points in $B(2,L,2)$. Since $b(1,2)=1$, the $3$ points must coincide. It is also easy to see directly that the $3$ bundles are the same.

More generally, if $L$ is a general line bundle of degree $d\ge3$ such that $L$ possesses a section with distinct zeroes, then $B(2,L,k) \neq \emptyset$ by \cite[Corollary 3.8]{gn} or \cite[Theorem 1.3]{o2}. The general $L$ has this property if and only if $d \geq g$. If $d<g$ and $L$ is general, then $h^0(L)=0$, so no extension of the form \eqref{eqb3} can exist. Hence, if $E\in B(2,L, 2)$, then $E$ possesses a line subbundle $M$ with $h^0(M)\ge2$. On a general curve, this implies $\deg M\ge\frac{g}2+1$, contradicting the stability of $E$. So $B(2,L,2)=\emptyset$ for $d<g$. When $d=2g-1-2r$, the condition $d<g$ is equivalent to $g\le2r$, so $P_2(\alpha,\beta,\gamma)\in I_{2r}$, in other words (see \eqref{eq2.1})
\begin{equation}\label{eqb5}
c_{2r+1}^2-c_{2r}c_{2r+2}\in I_{2r}.
\end{equation}
Even for $r=3$, the computation for \eqref{eqb5} is substantially more complicated than the one at the end of Example \ref{ex1}. Note that Theorem \ref{thm3.7} requires $g\ge2r+4$ for non-emptiness rather than $g\ge2r+1$.}
\end{ex}

\begin{ex}\label{ex3}
{\rm Let $k=3$. When $r=1$, we have $g^0_{1,3}=5$ and $d=7$. Again $b(1,3)$ is top dimensional and equal to $1$ (using Maple). Let $L$ be a generated line bundle of degree $7$ with $h^0(L)=3$ (this is true generically) and consider the bundle $E_L$ defined by the evaluation sequence
\begin{equation}\label{eqb4}
0\lra E_L^*\lra H^0(L)\otimes{\mathcal O}\lra L\lra0.
\end{equation}
If $C$ has Clifford index $2$, then $E_L$ is stable. In fact, if $M$ is any quotient line bundle of $E_L$, then $M$ is generated and $h^0(M^*)=0$. So $h^0(M)\ge2$ and hence $\deg M\ge4$. In order to show that $B(2,L,3)$ consists of one point, it remains to show that there are no bundles $E\in B(2,L,3)$ which are not generated. Certainly $E$ is generically generated since it cannot have a line subbundle of degree $\ge4$ and hence no subbundle with $h^0\ge2$. Let $E'$ be the subsheaf of $E$ generated by its sections and suppose that $\deg E'\le6$. Since $h^0(E')\ge3$, we can choose a $3$-dimensional subspace $V$ of $H^0(E')$ which generates $E'$. Dualising the evaluation sequence
$$0\lra\det E'^*\lra V\otimes {\mathcal O}\lra E'\lra0,$$
and noting that $h^0(E'^*)=0$, we see that $h^0(\det E')=3$ and hence $\deg E'\ge6$. So $\deg E'=6$ and $\det E'=L(-p)$ for some $p$. Since $\dim B(1,6,3)=2$, this is impossible for general $L$.

More generally, for $g\ge3$, $B(2,L,3) \neq \emptyset$ for all $L$ of degree $d$ if $d \geq g+2$ or equivalently $\beta(2,d,3) -g \geq 0$. For general $L$ on a general curve, this follows from \cite[Corollary 3.11 and Remark 3.12]{gn} (for stability in the case of even degree, see \cite[Proposition 4.6]{br}), and then, for any $L$ on any curve, by semi-continuity. If $C$ and $L$ are general and $d\le g+1$, then $h^0(L)=2$, so no extension similar to \eqref{eqb4} can exist. One can also rule out the possibility that $E\in B(2,L,3)$ is only generically generated or possesses a line subbundle with $h^0\ge3$. So, in general, $B(2,L,3)=\emptyset$ for $d\le g+1$, although it can certainly be non-empty for special $C$ and $L$. When $d=2g-1-2r$, the condition $d\ge g+2$ is equivalent to $g\ge2r+3$, while Theorem \ref{thm3.7} requires $g\ge2r+6$.} 
\end{ex}

\begin{ex}\label{ex4}
{\rm The case $k=4$, $r=1$ is particularly interesting. Here $g^0_{1,4}=8$, $d=13$ and the expected dimension of the Brill-Noether locus is $1$. It turns out that  $(\alpha\cdot P_1(\alpha,\beta,\gamma))$ is equal to $13$. This proves firstly that $B(2,L,4)$ is non-empty, which was not previously known (neither Section \ref{main} nor \cite{te2} applies). Secondly, recall that the unique line bundle on $M(2,L)$ with $c_1=\alpha$ is very ample \cite{bv}. One might therefore expect that, for general $L$, $B(2,L,4)$ is a curve whose degree with respect to this line bundle is 13. The construction of bundles $E\in B(2,L,4)$ is much harder than for the cases considered above ($k\le3$). There is, however, one method that should give a $1$-parameter family of such bundles. Let $C$ be a general curve of genus $8$ and $L$ a general line bundle of degree $13$ on $C$; in particular, $L$ is generated with $h^0(L)=6$. Consider the canonical map
$$\psi:S^2H^0(L)\lra H^0(L^2),$$
whose kernel is the Koszul cohomology group $K_{1,1}(C,L)$. We have $h^0(L^2)=19$ by Riemann-Roch and $\dim S^2H^0(L)=21$. For any non-zero element of $K_{1,1}(C,L)$, one can construct a rank $2$ bundle $E$ with determinant $L$ and $h^0(E)\ge4$ using \cite[Theorem 3.4]{an} and it can be shown that in general $E$ is generated and stable. 

This construction can be carried out in a more geometrical fashion by the method used in the proof of \cite[Theorem 3.2(ii)]{gmn}. Let $\phi_L:C\to{\mathbb P}^5={\mathbb P}(h^0(L)^*)$ be the morphism defined by evaluation of sections of $L$. The fact that $\dim\Ker\,\psi\ge2$ means that $\phi_L(C)$ is contained in a pencil of quadrics. If we choose a $3$-dimensional subspace $W$ of $H^0(L)$ such that the plane in ${\mathbb P}^5$ orthogonal to $W$ lies on one of the quadrics and does not meet $\phi_L(C)$, then $W$ generates $L$ and we can define $E$ by the evaluation sequence
$$0\lra E^*\lra W\otimes{\mathcal O}\lra L\lra0.$$
Clearly $E$ is generated. One can check firstly that $h^0(E)\ge4$ and then that $E$ is stable and $h^0(E)=4$. Dimensional calculations suggest that this should give a $1$-parameter family of bundles $E\in B(2,L,4)$. Whether this is the whole of $B(2,L,4)$ requires further investigation.

If $k=4$, $r=2$, we have from Remark \ref{rem3.8} that $b(2,4)\ne0$ if and only if $g\ge11$.}
\end{ex}

\begin{ex}\label{ex5}
{\rm For $k=5$, we have $g^0_{1,5}=11$ and $b(1,k)$ is a top dimensional class numerically equal to $23$. In this case, we already know that $b(1,k)\ne0$ (see Remark \ref{rem3.8}). A full investigation of the geometry is likely to be complicated.}
\end{ex}

\section*{Appendix. Proof of Lemma \ref{lem4.4}}
\renewcommand{\thesection}{A}
\setcounter{theorem}{0}

\vspace{0.2cm}
For any integer $r \geq 1$, consider the sequence of polynomials with integer coefficients
	${\tilde c}(n,r,b)$ defined recursively by
$$
	\tilde{c}(0,r,b) = 1, \quad  \tilde c(1,r,b) = r
$$
and for $n \geq 2$,
$$
\tilde c(n,r,b) = (r+1-n)\tilde c(n-1,r,b) + b(2r+1-n)(n-1)\tilde c(n-2,r,b).
$$

\begin{lem} \label{lem7.1}
 Lemma \ref{lem4.4} is a consequence of the following equation,
	\begin{equation}\label{eq7.2}
	\tilde{c}(2r,r,b) = 
	\prod_{j=1}^r 
	\left[ (2j-1)^2 \cdot b - j(j-1) \right].
	\end{equation}
\end{lem}

\begin{proof}
Inserting $b = \frac{1-\beta}{4}$ we have
$$
\tilde c(2r,r,b) = \frac{1}{2^{2r}} \prod_{j=1}^r (1 - (2j-1)^2\beta).
$$
On the other hand, $\tilde c(n,r,b) = n! c_{n}(1,\beta,0)$, since both sides satisfy the same 
recurrence relation and have the same initial values. 
\end{proof}
For the proof of \eqref{eq7.2} we need some preliminaries. Consider the following matrix.

\vspace{0.3cm}
\noindent
$D(2n;z,a)  :=$
$$ 
{\footnotesize \left( \begin{array}{ccccccc}
            z(a+n) & 2n -1 & 0 & 0 & \ldots & 0 & 0\\
            1 & z(a+n-1) & 2n-2 & 0 & \ldots & 0 & 0\\
           0 & 2 & z(a+n-2) & 2n- 3 & \ldots & 0 & 0\\
	    \vdots & \vdots & \ddots & \ddots & \ddots & \vdots & \vdots\\
	    0& 0  & 0 & 0 & 2n-2 &z(a-n+2) & 1\\
	    0 & 0 & 0 & 0 & 0 & 2n-1 & z(a-n+1)
             \end{array} \right). }        
$$  
We claim that for the proof of \eqref{eq7.2} it suffices to show that
\begin{equation} \label{eq7.1}
	\det D(2 n;z,a) = 
	\prod_{j=1}^n
	\left[
	z^2(a+j)(a-j+1) - (2j-1)^2 \right].
\end{equation}

\begin{proof}[Proof of the claim]
For this consider 3-band-matrices
 $$
	\widetilde{C}(n,r,b) = 
	\left( \widetilde{C}(n,r,b)_{i,j}
	\right)_{1 \leq i,j \leq n},
$$
where
\[
	\widetilde{C}(n,r,b)_{i,j} =
	\begin{cases} 
	r+1-i & \text{if}~j=i~(1 \leq i \leq n)
	\cr
	\sqrt{b} \cdot (i-1) & 
	\text{if}~j=i-1~(1<i\leq n)
	\cr
	\sqrt{b}\cdot (2r-i)
	& 
	\text{if}~j=i+1~(1\leq i < n)
	\cr
	0 &\text{otherwise}.
	\end{cases}
\]
For $1 \leq k < n$,
the matrix $\widetilde{C}(k,r,b)$
is the principal sub-minor of
size $k \times k$ of
$\widetilde{C}(n,r,b)$
(taking elements
in the first $k$ rows and columns).
Hence, for $n\ge4$, 
\begin{align*}
\widetilde{C}(n,r,b) &=
\left(
\begin{array}{cc|c}
\widetilde{C}(n-1,r,b) &   & {\bf 0}
\cr
 &   &\sqrt{b}(2r+1-n) 
\\\hline
\bf 0 & \sqrt{b} (n-1) & r+1-n
\end{array}
\right)
\cr
&=
\left(
\begin{array}{cc|cc}
\widetilde{C}(n-2,r,b) & & {\bf 0}  & {\bf 0}
\cr
 & & \sqrt{b}(2r+2-n) & 0
\\\hline
{\bf 0} & \sqrt{b}(n-2)  & r+2-n &\sqrt{b}(2r+1-n) 
\cr
\bf 0 & 0 & \sqrt{b} (n-1) & r+1-n
\end{array}
\right),
\end{align*}
and expanding $\det \widetilde{C}(n,r,b)$
starting from the lower right corner
of the matrix gives
\[
\det \widetilde{C}(n,r,b)
=
(r+1-n) \det \widetilde{C}(n-1,r,b)
-b(n-1)(2r+1-n)
\det \widetilde{C}(n-2,r,b),
\]
which is
 the recurrence relation defining $\tilde c(n,r,-b)$. Checking initial values
then gives
\[
\tilde{c}(n,r,b) = \det \widetilde{C}(n,r,-b)
\]
by induction.

Comparing the matrices $\widetilde{C}(2n,n,b)$ and $D(2n,b^{-\frac{1}{2}},0)$ one sees that 
$$
 \widetilde{C}(2n,n,b) = b^{\frac{1}{2}} \cdot  D(2n,b^{- \frac{1}{2}},0).
$$
and hence using \eqref{eq7.1},
\begin{align*}
\det \widetilde{C}(2n,n,b)
&= b^n \cdot \det D(2n,b^{-\frac{1}{2}},0)
\cr
&= b^n \cdot \prod_{1 \leq i \leq n}
\left[
\frac{1}{b}\cdot j \cdot (-j+1) - (2j-1)^2
\right]
\cr
&=
\prod_{1 \leq i \leq n}
\left[
 j \cdot (-j+1) - b \cdot (2j-1)^2
\right]
\end{align*}
and finally
\[
\tilde{c}(2n,n,b) = 
\det \widetilde{C}(2n,n,-b) =
\prod_{1 \leq i \leq n}
\left[
 b \cdot (2j-1)^2 -j \cdot (j-1)
\right].
\]
\end{proof}

It remains to prove \eqref{eq7.1}. For this  
consider the following matrices:
\begin{enumerate}
\item
$$
	A_n =  {\footnotesize \left( \begin{array}{ccccccc}
	1 & 0 & 0 & 0 & \ldots & 0 & 0\\
	1 & 2 & 0 & 0 & \ldots & 0 & 0\\
	0 & 2 & 3 & 0 & \ldots & 0 & 0\\
	\vdots & \vdots & \vdots & \vdots &
	\ddots & \vdots & \vdots \\
	0 & 0 & 0 & 0 & \ldots & n-1 & 0\\
	0 & 0 & 0 & 0 & \ldots & n-1 & n
	\end{array} \right)}.
$$	

\noindent
$A_n$ has simple eigenvalues
	$1,2,\ldots,n$ and the matrix
	of (non-orthogonal) left eigenvectors (i.e. $A_n$ is multiplied from the right)
	is the binomial matrix
	\[
	\left( {i-1 \choose j-1} \right)_{i,j=1,\dots,n}.
\]  
\item
	\[
	B_n =
	 {\footnotesize \left( \begin{array}{ccccccc}
	2 & -(n-1) & 0 & 0 &  \ldots & 0 & 0 \cr
	1 & 4 & -(n-2) & 0 &  \ldots & 0 & 0 \cr
	0 & 2 & 6 & -(n-3) &  \ldots & 0 & 0 \cr
	\vdots & \vdots &  \vdots &
	\vdots &
	\ddots & \vdots & \vdots \cr
	0 & 0 & 0 & 0 &  \ldots & 2(n-1) & -1\cr 
	0 & 0 & 0 & 0 &  \ldots &n-1 & 2n
	\end{array} \right) }.
	\]
	$B_n$ has $n+1$ as its only
	eigenvalue, which is $n$-fold
	and maximally degenerate (i.e.,
	the eigenspace of $n+1$ is
	one-dimensional).
 	The eigenspace is given by the vector	
$$ 
{\footnotesize
	\left(
	\binom{n-1}{0} , \binom{n-1}{1} ,
	\binom{n-1}{2} , \ldots , \binom{n-1}{n-1}
	\right)  }
$$
	which is indeed the same as the last
	eigenvector of $A_n$.
\end{enumerate}

 There is an intimate relation between 	$A_n$ and $B_n$: the matrix
	$\widetilde{B}_n=(n+1)I_n-B_n$.
	maps the eigenvectors of $A_n$
	as follows. For $1 \leq k \leq n-1$, let 
\[ {\small
	\bs{\alpha}_k := 
	\left(
	\binom{k-1}{0} , \binom{k-1}{1} , \binom{k-1}{2} ,\ldots, \binom{k-1}{n-1} \right) }
\]
	denote the left eigenvectors of $A_n$.
	Then, for $1 \leq k \leq n-1$,
	\[
	\bs{\alpha}_k \mapsto 
	\bs{\alpha}_k \, \widetilde{B}_n
	=(n-k)\,\bs{\alpha}_{k+1} 
	\]
	and (obviously)
	\[
	\bs{\alpha_n} \mapsto
	\bs{\alpha_n} \,\widetilde{B}_n = \bs{0}.
	\]

\begin{lem} \label{lem7.2}
Let $s$ be a real parameter. The characteristic polynomial of the matrix
$C_n(s) := (n+1)\cdot A_n +s \cdot B_n$ is
\[
	\chi(C_n(s);z) =
	\prod_{1 \leq k \leq n}
	(z - (s+k)(n+1))
	\]
\end{lem}

\begin{proof} For the proof we consider the matrix of the transformation 
	given by $C_n(s)$ in the eigenbasis   
	$\bs{\alpha}_1, \ldots , \bs{\alpha_n}$ of $A_n$. We have
$$
	C_n(s) =(n+1)(A_n+s\,I_n)-s \widetilde{B}_n
$$
	and hence, for $1 \leq k \leq n$,
	\[
	C_n(s) : \bs{\alpha}_k
	\longmapsto
	(n+1)(k+s) \,\bs{\alpha}_k -s(n-k)\, \bs{\alpha}_{k+1}.
	\]
	This shows that the transformation
	$C_n(s)$ in the basis  
	$\bs{\alpha}_1, \ldots , \bs{\alpha_n}$
	is
	\[
{\footnotesize	
 \left( \begin{array}{cccccc}
	(n+1)(1+s) & -s(n-1) & 0 & 0 & \ldots & 0
	\cr
	0 & (n+1)(2+s) & -s(n-2) & 0 & \ldots & 0
	\cr
	\vdots & \ddots & \ddots & \ddots & \ldots & \vdots 
	\cr
	0 & 0 & 0 & 0 &  \ldots & -s
	\cr
	0 & 0 & 0 & 0 & \ldots &(n+1)(n+s)
	\end{array} \right)}
	\]
	and the eigenvalues are simply
	the diagonal elements 
	$(n+1)(s+k)~(1 \leq k \leq n)$. This implies the assertion.
\end{proof}

	From now on we assume that
	$n$ is even.

\begin{lem}  \label{lem7.3}
The characteristic polynomial of the matrix
	\[
	L_n(a)=
{\footnotesize	\left(
\begin{array}{ccccccc}
	a & n-1 & 0 & 0 & 0 & \ldots & 0  
	\cr
	1 & 2\,a & n-2 & 0 & 0 & \ldots & 0 
	\cr
	0 & 2 & 3\,a & n-3  & 0 & \ldots & 0
	\cr
	\vdots & \vdots &  \ddots &
	\ddots & \ddots & \ldots & 0
	\cr
	0 & 0 & \ldots & \ldots &n-2 & a\,(n-1) & 1
	\cr
	0 & 0 & \ldots & \ldots &0 & n-1 & a\,n
	\end{array} \right) }
	\]
	($a$ is a parameter) is
$$
	\Lambda_n(a;z):=
	\prod_{i+j=n+1\atop 1 \leq i < j \leq n}
	\left[
	(z-i\cdot a)(z-j \cdot a)
	- (i-j)^2
	\right].
$$
\end{lem}

\begin{proof}
Let  $a_\ell=\frac{k-\ell}{\sqrt{k\,\ell}}$ 
	where $k,\ell$ are real parameters.
	Multiplying
	by $\sqrt{k\,\ell}$ gives
	\begin{multline*}
	\hbox{$ 
	\sqrt{k\,\ell} \cdot L_n(a_\ell)$} 
	=
	\cr
	{\footnotesize  \left( \begin{array}{ccccccc}
	k-\ell & (n-1)\sqrt{k\,\ell} & 0 & 0 & 0 & \ldots & 0  
	\cr
	\sqrt{k\,\ell} & 2\,(k-\ell) & (n-2)\sqrt{k\,\ell} & 0 & 0 & \ldots & 0 
	\cr
	0 & 2\sqrt{k\,\ell} & 3\,(k-\ell) & (n-3)\sqrt{k\,\ell}  & 0 & \ldots & 0
	\cr
	\vdots & \vdots &  \ddots &
	\ddots & \ddots & \ldots & 0
	\cr
	0 & 0 & \ldots & \ldots &(n-2)\sqrt{k\,\ell} & (n-1)(k-\ell) &  \sqrt{k\,\ell}
	\cr
	0 & 0 & \ldots & \ldots &0 & (n-1)\sqrt{k\,\ell} & n(k-\ell)
		\end{array} \right) . }
	\end{multline*}
	
	This matrix is similar to
	\[
{\footnotesize \left(
	\begin{array}{ccccccc}
	k-\ell & (n-1)\ell & 0 & 0 & 0 & \ldots & 0  
	\cr
	k & 2\,(k-\ell) & (n-2)\ell & 0 & 0 & \ldots & 0 
	\cr
	0 & 2\,k & 3\,(k-\ell) & (n-3)\ell  & 0 & \ldots & 0
	\cr
	\vdots & \vdots &  \ddots &
	\ddots & \ddots & \ldots & 0
	\cr
	0 & 0 & \ldots & \ldots &(n-2)\,k & (n-1)(k-\ell) &  \ell
	\cr
	0 & 0 & \ldots & \ldots &0 & (n-1)\,k & n(k-\ell)
	\end{array}  \right)   }
	\]
	If $k$ and $\ell$ are related by
	$k+\ell=n+1$, then this is just
	the matrix $C_n(-\ell)$ 
	we considered above. 

For the rest of the proof we assume now $k+ \ell = n+1$. 
	 Then  Lemma \ref{lem7.2} implies that the charactristic polynomial of
	$\sqrt{k\,\ell} \cdot L_n(a_\ell)$ is
	$\prod_{1 \leq m \leq n} 
	\left[z-(n+1)(m-\ell)\right]$ which immediately gives:
       the matrix $L_n(a_\ell)$ has characteristic polynomial
	\[
	\prod_{1 \leq m \leq n} 
	\left[z-(n+1)
	\frac{m-\ell}{\sqrt{k\,\ell}}
	\right].
	\]

Now $\Lambda_n(a;z)$
	is a monic polynomial of degree $n$ in
	the parameter $a$.
	For showing that 
	$\Lambda_n(a;z)$ is the characteristic
	polynomial of 
	$L_n(a)$, it suffices to show
	that for the $n$ interpolation points
	$a_\ell~(1 \leq \ell \leq n)$
	the polynomial
	$\Lambda_n(a_\ell;z)$ 
	is indeed the characteristic
	polynomial of
	$L_n(a_\ell)$, i.e. that, for $1 \leq \ell \leq n$,
	\[
	\Lambda_n(a_\ell;z) =
	\prod_{1 \leq m \leq n} 
	\left[z-(n+1)
	\frac{m-\ell}{\sqrt{k\,\ell}}
	\right].
	\]
	Now both sides are monic polynomials
	of degree $n$ in $z$.
	Since the expression on the right hand side
	vanishes at the $n$ interpolation
	points
	\[
	\xi_{\ell,m} := (n+1)
	\frac{m-\ell}{\sqrt{k\,\ell}}~~
	(1 \leq m \leq n),
	\]
	it suffices to show that for $1 \leq m \leq n$,
	\[
	\Lambda_n(a_\ell;\xi_{\ell,m})=0.
	\]
	We write explicitly
	\begin{multline*}
	\Lambda_n(a_\ell;\xi_{\ell,m})=\cr
	{\footnotesize \prod_{i+j=n+1 \atop 1 \leq i < j \leq n}
	\left[
	(n+1)^2 \frac{(m-\ell)^2}{k\,\ell}
	-(n+1)^2 \frac{(m-\ell)(k-\ell)}{k\,\ell} + i\,j \frac{(k-\ell)^2}{k\,\ell}
	-(i-j)^2
	\right]. }
	\end{multline*}
	We have to show that (at least)
	one of the bracketed terms under
	the product vanishes.
	Now we use both
	conditions 
	$k+\ell=n+1$ and $i+j=n+1$
	crucially!
	
	From
	\[
	(k-\ell)^2 =  (n+1)^2-4k\ell
	\]
	we have
	\begin{align*}
	ij \frac{(k-\ell)^2}{k\ell}
	- (i-j)^2
	&=
	ij \frac{(n+1)^2}{k\ell} - 4ij
	-(i^2-2ij+j^2)
	\cr
	&
	=
	ij \frac{(n+1)^2}{k\ell}-(n+1)^2
	\cr
	&= (n+1)^2 
	\left(\frac{ij}{k\ell} -1 \right)
	\end{align*}
	and (since $n$ is even)
	\begin{align*}
	\Lambda_n(a_\ell,\xi_{\ell,m})
	&=
	\left(\frac{n+1}{\sqrt{k \ell}}
	\right)^{n}
	\prod_{i+j=n+1 \atop 1 \leq i < j \leq n}
	\left[ 
	(m-\ell)^2-(m-\ell)(k-\ell)+ij-k \ell 
	\right]
	\cr
	&=
	\left(\frac{n+1}{\sqrt{k \ell}}
	\right)^n
	\prod_{i+j=n+1 \atop 1 \leq i < j \leq n}
	\left[ 
	m^2-(k+\ell)m+ij
	\right]
	\cr
	&=
	\left(\frac{n+1}{\sqrt{k \ell}}
	\right)^n
	\prod_{i+j=n+1 \atop 1 \leq i < j \leq n}
	\left[ 
	m^2-(i+j)m+ij
	\right]
	\cr
	&=
	\left(\frac{n+1}{\sqrt{k \ell}}
	\right)^n
	\prod_{i+j=n+1 \atop 1 \leq i < j \leq n}
	\left[ 
	(m-i)(m-j)
	\right].
	\end{align*}
	This shows that for all 
	$1 \leq m \leq n$ and
	$z=\xi_{\ell,m}$
	the quadratic factor
	belonging to $m=i$ and $m=j$
	vanishes.
\end{proof}

	Lemma \ref{lem7.3} 
	can be stated in the
	following (elegant) way (again, for even $n$).
	 For parameters $a,b$, let
	$\varepsilon_{a,b}$ be the affine function
	$\varepsilon_{a,b}(x):=a \cdot x + b$.
	Then {\small
	\[
	\det \left(
	\begin{array}{ccccc}
	\varepsilon_{a,b}(1) & n-1 & 0 & 0 &\ldots\cr
	1 & \varepsilon_{a,b}(2) & n-2 & 0 &
	\ldots \cr
	0 & 2 & \varepsilon_{a,b}(3) & n-3 &
	\ldots \cr
	\vdots & \vdots & \ddots & \ddots &
	\vdots
	\cr
	0 & \ldots &n-2 & \varepsilon_{a,b}(n-1) & 1
	\cr
	0 & \ldots & 0 &
	n-1 & \varepsilon_{a,b}(n)
	\end{array} \right)
	=
	\prod_{1 \leq i < j \leq n \atop i+j=n+1}
	\left[
	\varepsilon_{a,b}(i) \cdot
	\varepsilon_{a,b}(j) -(i-j)^2
	\right].
	\] }
	
Now we are in a position to complete the proof of \eqref{eq7.1}.
	
\begin{proof}[Proof of equation \eqref{eq7.1}]
 For $n$ even we have
	to show 
$$
	\det {\footnotesize \left(
	\begin{array}{cccccc}
	z(a+n/2) & n -1 & 0 &  \ldots & 0 & 0
	\cr
	1 & z(a+n/2-1) & n-2 
	& \ldots & 0 & 0
	\cr
	0 & 2 & z(a+n/2-2) 
		& \ldots & 0 & 0
	\cr
	\vdots & \vdots & \ddots 
	& \ddots & \vdots & \vdots
	\cr
	0& 0  & 0 & 0 &z(a-n/2+2) & 1
	\cr
	0 &
	0 & 0 & 0 & n-1 & z(a-n/2+1)
	\end{array}  \right)    }
$$
$$
= \prod_{j=1}^{n/2}
	\left[
	z^2(a+j)(a-j+1) - (2j-1)^2 \right]
$$
	
	Replacing $a$ by $a+n/2$, the matrix becomes
$$
{\footnotesize \left(
	\begin{array}{cccccc}
	z(a+n) & n -1 & 0  
	& \ldots & 0 & 0
	\cr
	1 & z(a+n-1) & n-2 
	& \ldots & 0 & 0
	\cr
	0 & 2 & z(a+n-2) 
		& \ldots & 0 & 0
	\cr
	\vdots & \vdots & \ddots 
	& \ddots & \vdots & \vdots
	\cr
	0 & 0 & 0&n-2 &z(a+2) & 1
	\cr
	0 &
	0 & 0 & 0 & n-1 & z(a+1)
	\end{array}, \right)  }
$$
	and this is (up to reordering
	rows and columns)
	\[
{\footnotesize \left(
	\begin{array}{ccccc}
	\varepsilon_{z,az}(1) & n-1 & 0 & 0 &\ldots\cr
	1 & \varepsilon_{z,az}(2) & n-2 & 0 &
	\ldots \cr
	0 & 2 & \varepsilon_{z,az}(3) & n-3 &
	\ldots \cr
	\vdots & \vdots & \ddots & \ddots &
	\vdots
	\cr
	0 & \ldots &n-2 & \varepsilon_{a,az}(n-1) & 1
	\cr
	0 & \ldots & 0 &
	n-1 & \varepsilon_{z,az}(n)
	\end{array} \right).   }
	\]
According to Lemma \ref{lem7.3} the determinant of the last matrix is
\begin{align*}
	&\prod_{1 \leq i < j \leq n \atop i+j=n+1}
	\left[
	\varepsilon_{z,az}(i) \cdot
	\varepsilon_{z,az}(j) -(i-j)^2
	\right] \cr
	&=\prod_{1 \leq i < j \leq n \atop i+j=n+1}
	\left[
	(z\cdot i+a\cdot z)(z \cdot j+a \cdot z)-(i-j)^2
	\right]
	\cr
	&=
	\prod_{1 \leq i < j \leq n \atop i+j=n+1}
	\left[
	z^2(a^2 + (i+j)a + ij)-(i-j)^2
	\right]
	\cr
	&=
	\prod_{1 \leq i \leq n/2} 
	\left[
	z^2 (a+i)(a+n-i+1)-(n+1-2i)^2
	\right]
	\cr
	&=
	\prod_{1 \leq i \leq n/2} 
	\left[
	z^2(a+n/2-i+1)(a+n/2+i)-(2i-1)^2
	\right].
	\end{align*}
Replacing back $a + \frac{n}{2}$ by $a$, this gives the assertion.
\end{proof}



\begin{thebibliography}{CAV}

\bibitem{an}
M. Aprodu and J. Nagel:
\emph{Non-vanishing for Koszul cohomology of curves}.
Comment. Math. Helv. 82 (2007), 617--628.
\bibitem{acgh}
E. Arbarello, M. Cornalba, P. A. Griffiths and J. Harris: 
\emph{Geometry of Algebraic Curves I}. 
Springer, Grundlehren math. Wiss. 267 (1985).
\bibitem{br} L. Brambila-Paz:
\emph{Non-emptiness of moduli spaces of coherent systems}.
Internat. J. Math. 19 (2008), 777--799.
\bibitem{bmno} L. Brambila-Paz, V. Mercat, P. E. Newstead and F. Ongay:
\emph{Non-emptiness of Brill-Noether loci}. 
Internat. J. Math. 11 (2000), 737--760.
\bibitem{bv} S. Brivio and A. Verra:
\emph{On the theta divisor of $\mbox{SU}(2,1)$}.
Internat. J. Math. 10 (1999), 925--942.
\bibitem{cf1} C. Ciliberto and F. Flamini:
\emph{Brill-Noether loci of stable rank-two vector bundles on a general curve}.
In  C. Faber, G. Farkas and R. de Jong (eds), Geometry and Arithmetic, Series of Congress Reports, European Mathematical Society, Z\"urich, 2012, 61--74.
\bibitem{cf2} C. Ciliberto and F. Flamini:
\emph{Extensions of line bundles and Brill-Noether loci of rank-two on a general curve}.
ArXiv 1312.6239v3, to appear in Revue Roumaine de Math\'ematiques Pures et Appliqu\'ees..
\bibitem{fh} 
W. Fulton and J. Harris: 
\emph{Representation theory: A first course}.
Graduate Texts in Math. 129, Springer (1991).
\bibitem{gmn}
I. Grzegorczyk, V. Mercat and P. E. Newstead:
\emph{Stable bundles of rank $2$ with four sections}.
Internat. J. Math. 22 (2011), 1743--1762.
\bibitem{gn}
I. Grzegorczyk and P. E. Newstead:
\emph{On coherent systems with fixed determinant}.
Internat. J. Math. 25 (2014), no.5, 1450045, 11pp.
\bibitem{kn}
A. D. King and P. E. Newstead:
\emph{On the cohomology ring of the moduli space of rank 2 vector bundles on a curve}.
Topology 37 (1997), 567--577.
\bibitem{lnp}
H. Lange, P. E. Newstead and Seong Suk Park:
\emph{Non-emptiness of Brill-Noether loci in $M(2,K)$}.
ArXiv 1311.5007v3, to appear in Comm. in Algebra.
\bibitem{n2}
P. E. Newstead:
\emph{Stable bundles of rank $2$ and odd degree over a curve of genus $2$}.
Topology 7 (1968), 205--215.
\bibitem{n}
P. E. Newstead:
\emph{Characteristic classes of stable bundles of rank 2 over an algebraic curve}.
Trans. Amer. Math. Soc. 169 (1972), 337--345. 
\bibitem{o}
B. Osserman:
\emph{Brill-Noether loci with fixed determinant in rank 2}.
 Internat. J. Math. 24 (2013), no.13, 1350099, 24pp. 
\bibitem{o2} B.~Osserman:
\emph{Special determinants in higher rank Brill-Noether theory}.
Internat. J. Math. 24 (2013), no.11, 1350084, 20pp.
\bibitem{ram}
S. Ramanan:
\emph{The moduli spaces of vector bundles over an algebraic curve}.
Math. Ann. 200 (1973), 69--84.
\bibitem{sun}
N. Sundaram:
\emph{Special divisors and vector bundles}.
T\^{o}hoku Math. J. 39 (1987), 175--213.
\bibitem{te1} M. Teixidor i Bigas:
\emph{Brill-Noether theory for stable vector bundles}.
Duke Math. J. 62 (1991), 385--400.
\bibitem{te3} M. Teixidor i Bigas:
\emph{Existence of vector bundles of rank two with sections}.
Adv. Geom. 5 (2005), 37--47.
\bibitem{te2} M.~Teixidor i Bigas:
\emph{Existence of vector bundles of rank two with fixed determinant and sections}.
Proc. Japan Acad. 86 (2010), Ser. A, no.7, 113--118.
\bibitem{th}
M. Thaddeus:
\emph{Conformal field theory and the cohomology of the moduli space of stable bundles}.
J. Diff. Geom. 35 (1992), 131--149.
\bibitem{zh}
N. Zhang: 
\emph{Expected dimensions of higher-rank Brill-Noether loci}.
ArXiv 1312.1657v4.


\end{thebibliography}
\end{document}